\documentclass[11pt,letterpaper]{amsart}
\usepackage{amsmath,amssymb,mathtools,booktabs}
\usepackage[shortlabels]{enumitem}
\usepackage[hidelinks]{hyperref}
\usepackage{url}

\newtheorem{theorem}{Theorem}[section]
\newtheorem{proposition}[theorem]{Proposition}
\newtheorem{lemma}[theorem]{Lemma}
\newtheorem{corollary}[theorem]{Corollary}
\theoremstyle{definition}
\newtheorem{definition}[theorem]{Definition}
\theoremstyle{remark}
\newtheorem{remark}[theorem]{Remark}
\newtheorem{example}[theorem]{Example}
\newcommand{\Z}{\mathbb Z}
\newcommand{\Q}{\mathbb Q}
\newcommand{\F}{\mathbb F}
\newcommand{\C}{\mathbb C}
\newcommand{\A}{\mathbb A}
\newcommand{\Nm}{\operatorname{N}}
\newcommand{\disc}{\operatorname{disc}}
\newcommand{\Res}{\operatorname{Res}}
\newcommand{\rad}{\operatorname{rad}}
\newcommand{\Ht}{\operatorname{Ht}}
\newcommand{\ind}{\operatorname{ind}}

\newcommand{\cF}{\mathcal F}
\newcommand{\cO}{\mathcal O}
\renewcommand{\epsilon}{\varepsilon}

\allowdisplaybreaks[2]

\title[Galois groups with prescribed coefficients]{Galois groups of integer polynomials\break with prescribed coefficients}
\author{Evan M. O'Dorney}
\date{September 23, 2026. Research draft.}

\begin{document}
\begin{abstract} 
Fix a degree $n$, a nonzero integer constant term, and some further coefficients of a monic integer polynomial, leaving $s$ coefficients to vary in $[-H,H]$. We first prove that, when only the leading and constant coefficients are fixed, the number of polynomials whose Galois group is not $S_n$ is of order $H^{n-2}$ for every $n\ge4$. We then prove a half-dimensional extension: for $n\ge6$ the corresponding exceptional count is of order $H^{s-1}$ whenever $s>n/2+1$, provided the coefficient slice meets the higher-multiplicity loci with a specified dimension bound. This geometric condition holds for every prescribed value if the free coefficients form an initial or final block, and for a nonempty Zariski-open set of prescribed values for arbitrary coefficient positions. It also holds for every value of one additional prescribed interior coefficient. Thus, in degrees seven and eight, any one interior coefficient can be fixed in addition to the leading and constant coefficients.

The proof combines a fixed-norm version of Bhargava's primitive-group sieve with sparse Hermite interpolation. A separate imprimitive estimate, using a minimal intermediate field, gives $O(H^{\beta(n)})$ outside rational partial-product hypersurfaces, where $\beta(n)=\ell+n/\ell-2\le n/2$ for composite $n$ with smallest prime divisor $\ell$, and has no logarithmic loss. The fixed-constant quartic and quintic endpoints require separate local arguments: exceptional Fourier directions for $A_4$ quartics and a quadratic Gauss sum for quintics. We isolate further even-degree endpoint results, explain the distinct cubic obstruction, and conclude with conditional improvements under the upper-bound form of Malle's conjecture.
\end{abstract}
\maketitle

\section{Introduction}

For monic degree-$n$ integer polynomials with all $n$ nonleading
coefficients bounded by $H$, Bhargava proved the van der Waerden bound
$O_n(H^{n-1})$ for the number with Galois group different from $S_n$
\cite{Bhargava}. Cohen's large-sieve results allow prescribed coefficient
positions, under the appropriate generic-Galois-group hypotheses, with
an exceptional proportion $O(H^{-1/2}\log H)$ \cite{Cohen}. The question
considered here is whether the stronger $H^{-1}$ scale survives when
several coefficients are fixed.

Our first result treats the basic prescribed-coefficient family in
which only the leading coefficient and a nonzero constant term are
fixed. There are then $n-1$ free coefficients. The fixed constant term
is particularly useful because a root of an irreducible polynomial has
fixed norm; this reduces the number of generators contributed by any
one number field to a polylogarithmic quantity. The second ingredient
is a Fourier sieve carried out directly on the coefficient slice.
Neither an estimate for the whole coefficient box nor a count of
number fields can simply be divided by $H$ each time another
coefficient is prescribed.

We then ask how many additional coefficients may be fixed while
retaining the stronger $H^{-1}$ exceptional proportion. We prove a
half-dimensional extension. The distinction between coefficient
\emph{positions} and prescribed coefficient \emph{values} is
important: arbitrary positions are permitted, but an unrestricted
assertion for every prescribed tuple would require additional control
of special multiplicity loci. Our theorems make this qualification
explicit.

\subsection{Families and the main results}
Write
\begin{equation}\label{eq:family}
 f_{\mathbf a}(x)=x^n+\sum_{i\in I}a_i x^i
                         +\sum_{j\in J}b_jx^j+c,
 \qquad I\sqcup J=\{1,\ldots,n-1\},\quad |I|=s,
\end{equation}
where $c\in\Z\setminus\{0\}$ and $\mathbf b=(b_j)_{j\in J}$ are fixed.
Our indices are the exponents of $x$, rather than the distances from
the leading term. Let $\cF(H)$ consist of the polynomials
\eqref{eq:family} with $a_i\in\Z$ and $|a_i|\le H$. We write
$E_{\cF}(H)$ for the number that are reducible, have discriminant zero,
or are irreducible with Galois group properly contained in $S_n$.
Thus $|\cF(H)|=(2\lfloor H\rfloor+1)^s$. We use coefficient
height $\Ht(f)=\max(1,|c|,|a_i|,|b_j|)$. Since the prescribed
coefficients are fixed, for sufficiently large $H$ every member of
$\cF(H)$ has height at most $H$. Implied constants may depend on
$n,I,\mathbf b,c$ and on explicitly indicated auxiliary parameters.

When $J=\varnothing$, write $E_{n,c}(H)$ for $E_{\cF}(H)$. Thus
$E_{n,c}(H)$ counts monic degree-$n$ polynomials
\[
 x^n+a_{n-1}x^{n-1}+\cdots+a_1x+c,
 \qquad |a_i|\le H,
\]
that are reducible, inseparable, or have Galois group different from
$S_n$.

\begin{theorem}\label{thm:fixed-constant}
Fix $n\ge4$ and $c\in\Z\setminus\{0\}$. Then
\[
 E_{n,c}(H)\asymp_{n,c}H^{n-2}.
\]
Equivalently, among monic degree-$n$ integer polynomials with fixed
nonzero constant term $c$, the proportion whose Galois group is not
$S_n$ is of order $H^{-1}$.
\end{theorem}

Theorem~\ref{thm:main} below proves Theorem~\ref{thm:fixed-constant}
for $n\ge6$. Degrees four and five require the separate endpoint
arguments in Section~\ref{sec:edges}.

The family defines a coordinate slice $L\simeq\A^s$ in the space of
monic degree-$n$ polynomials with constant term $c$. For $1\le k<n$,
let $\Sigma_k$ be the closed locus of polynomials having at most $n-k$
distinct roots over an algebraic closure. Equivalently,
$\Sigma_k$ is the locus of polynomials of index at least $k$, where
\[
 \ind(f)=n-\deg\rad(f).
\]
All geometric dimensions below are taken over $\overline\Q$; the
empty set has dimension $-\infty$.

\begin{definition}\label{def:regular}
A slice $L$ is \emph{index-regular} if
\begin{equation}\label{eq:regular}
 \dim(L\cap\Sigma_k)\le\max(s-k,0)
                \qquad(2\le k\le n-1).
\end{equation}
It is \emph{strongly index-regular} if
$\dim(L\cap\Sigma_k)\le s-k$ for $k\le s$, and
$L\cap\Sigma_k=\varnothing$ for $k>s$.
\end{definition}

\begin{theorem}\label{thm:main}
Suppose $n\ge6$, $c\ne0$, and
\begin{equation}\label{eq:threshold}
 s>\frac n2+1,
 \qquad\text{equivalently}\qquad
 s\ge\left\lfloor\frac n2\right\rfloor+2.
\end{equation}
If the slice \eqref{eq:family} is index-regular, then
\begin{equation}\label{eq:main}
 E_{\cF}(H)\asymp_{n,I,\mathbf b,c} H^{s-1}.
\end{equation}
In particular the exceptional proportion is of order $H^{-1}$.
\end{theorem}

The geometric hypothesis has the following unconditional instances.

\begin{theorem}\label{thm:instances}
Fix $n$, $c\ne0$, and the coefficient positions $I,J$.
\begin{enumerate}[(a)]
\item If $I=\{1,\ldots,s\}$ or
$I=\{n-s,\ldots,n-1\}$, the slice is index-regular for every
$\mathbf b\in\Z^J$.
\item There is a nonempty Zariski-open subset
$U_{I,c}\subseteq\A^J_{\Q}$ such that the slice is strongly
index-regular for every $\mathbf b\in U_{I,c}(\Q)$.
\item If $|J|\le1$, the slice is index-regular for every prescribed
tuple and every choice of the prescribed positions.
\end{enumerate}
Consequently Theorem~\ref{thm:main} applies in (a) and (c) for every
integral prescribed tuple, and in (b) for every integral tuple in
$U_{I,c}$.
\end{theorem}

In (a), fixing the constant term and a consecutive block of leading
coefficients is allowed without a genericity qualification. In (b),
``generic'' refers only to the \emph{fixed} tuple: once an integral tuple
in $U_{I,c}$ is selected, \eqref{eq:main} is an unconditional estimate
for that particular family, not an average over prescribed values.

\begin{corollary}\label{cor:one-extra}
For $n\ge7$, fix $c\ne0$, a position $1\le j\le n-1$, and an
integer $b$. Among monic degree-$n$ polynomials with constant term $c$,
coefficient of $x^j$ equal to $b$, and all other coefficients bounded
by $H$, the number exceptional is $\asymp H^{n-3}$.
\end{corollary}

The leading coefficient may also be fixed at an arbitrary nonzero
integer: replacing $f(x)$, of leading coefficient $a$, by
$a^{n-1}f(x/a)$ produces a monic polynomial with constant term
$a^{n-1}c$ and rescales each free coefficient by a fixed nonzero
factor. The arguments below apply to such rectangular boxes and
fixed sublattices. We state the results in monic form to avoid extra
notation.

\subsection{The mechanism and its limitations}
There are three separate issues. First, prescribed coefficients must
not enlarge a multiplicity stratum too much. The capped dimension
bound \eqref{eq:regular}, rather than an unjustified density $p^{-k}$
in all codimensions, is sufficient for the sieve. Second, the
index-two locus must have cancellation in nonzero Fourier directions.
For arbitrary positions, a sparse Hermite interpolation argument
provides the stronger estimate $O(p^{-3})$ in the range
\eqref{eq:threshold}. Third, small-discriminant fields contribute
roughly $H^{n/2}$ with the field-counting estimate used here. This
last comparison explains the strict inequality $s-1>n/2$.

The imprimitive contribution does not force a larger constant in
\eqref{eq:threshold}. We show that, outside a bounded collection of
rational partial-product hypersurfaces, it is at most
\begin{equation}\label{eq:beta-intro}
 H^{\beta(n)},\qquad
 \beta(n)=\max_{\substack{d\mid n\\1<d<n}}
                     (d+n/d-2)\le n/2.
\end{equation}
The exponent has no logarithmic loss, which will matter at even-degree
endpoints. We make no claim that the field-counting bound used here
is optimal for every degree or group.

\subsection{Acknowledgments}
I thank Theresa ``Tess'' Anderson for useful discussions, original ideas, and writing that went into an earlier draft of this article,  \href{https://arxiv.org/abs/2603.15875}{arxiv:2603.15875} (\textsection3). As this version is heavily AI-generated and the results are noticeably stronger, I have decided to post it as a new arXiv article.

\section{Rational partial products and reducibility}\label{sec:products}

We begin with a hypersurface estimate that will also remove the
degenerate cases of the imprimitive argument.

\begin{lemma}\label{lem:partial-products}
Suppose
\begin{equation}\label{eq:gcd}
 \gcd\bigl(n,\{i:i\in I\}\bigr)=1.
\end{equation}
The number of polynomials in $\cF(H)$ for which a proper nonempty
submultiset of the $n$ roots has rational product is $O(H^{s-1})$.
In particular this bounds the reducible polynomials in the family.
\end{lemma}
\begin{proof}
A rational product of roots is an integer, since the roots are
algebraic integers. Its complementary product is also a rational
algebraic integer. Thus the product is a signed divisor of $c$.
For $1\le r<n$ and a fixed nonzero signed divisor $d$ of $c$, the
resolvent
\begin{equation}\label{eq:product-resolvent}
 R_{r,d}(f)=\prod_{\substack{S\subseteq\{1,\ldots,n\}\\|S|=r}}
                  \left(d-\prod_{\nu\in S}\alpha_\nu\right)
\end{equation}
is a polynomial with integer coefficients in the coefficients of
$f$. Its degree is bounded in terms of $n$. We claim that its
restriction to $L$ is nonzero.

Work over the rational function field in the free coefficients.
If \eqref{eq:product-resolvent} vanished identically, some fixed
$r$-element submultiset of the generic roots would have product $d$.
For $i\in I$, take the discrete valuation at $a_i=\infty$, normalized
by $v(a_i)=-1$, while the other parameters have valuation zero. The
Newton polygon of $f$ has vertices
\[
 (0,0),\quad(i,-1),\quad(n,0).
\]
Its roots have valuations $1/i$, occurring $i$ times, and
$-1/(n-i)$, occurring $n-i$ times. If the selected submultiset
contains $v_i$ roots of the first kind and $u_i$ of the second, then
\[
 \frac{v_i}{i}-\frac{u_i}{n-i}=0,
 \qquad u_i+v_i=r,
 \qquad\text{so}\qquad n v_i=ir.
\]
Hence $n\mid ir$ for every $i\in I$. Condition \eqref{eq:gcd}
forces $n\mid r$, a contradiction.

A nonzero polynomial of bounded degree in $s$ variables has
$O(H^{s-1})$ integer zeros in the box, by induction on $s$.
Apply this to the finitely many resolvents \eqref{eq:product-resolvent}.
If $f$ is reducible over $\Q$, a monic integral proper factor supplies
a rational partial product, proving the last assertion.
\end{proof}

\begin{remark}\label{rem:gcd-dense}
Condition \eqref{eq:gcd} is automatic under \eqref{eq:threshold}.
Indeed, a proper divisor $d\ge2$ of $n$ has at most
$\lfloor(n-1)/2\rfloor$ multiples in $\{1,\ldots,n-1\}$. It is
also automatic whenever $1\in I$.
\end{remark}

\begin{lemma}\label{lem:lower}
For every family \eqref{eq:family} with $s\ge2$, there are
$\gg H^{s-1}$ reducible polynomials in $\cF(H)$.
\end{lemma}
\begin{proof}
The equation $f(1)=0$ is
\[
 \sum_{i\in I}a_i=-1-c-\sum_{j\in J}b_j.
\]
Choose $s-1$ free coefficients in a sufficiently small fixed multiple
of $[-H,H]$ and solve for the last. For all sufficiently large $H$,
this gives $\gg H^{s-1}$ different polynomials divisible by $x-1$.
\end{proof}

\section{Multiplicity strata and coefficient slices}\label{sec:geometry}

The loci $\Sigma_k$ have equations and degrees bounded in terms of
$n$; they can be described either by subresultants or as finite
unions of multiplicity strata and their closures. We use the latter
description. For positive integers $e_1,\ldots,e_r$ with
$\sum_{\nu=1}^r e_\nu=n$, write
\begin{equation}\label{eq:root-stratum}
 f(x)=\prod_{\nu=1}^r(x-z_\nu)^{e_\nu},
 \qquad z_\nu\ne0,\quad z_\nu\ne z_\mu\ (\nu\ne\mu),
 \qquad (-1)^n\prod z_\nu^{e_\nu}=c.
\end{equation}
This is a bounded-degree finite parametrization of the corresponding
stratum. Its index is $k=n-r$, and its dimension before further
coefficients are prescribed is $r-1$.

\begin{lemma}\label{lem:coefficient-projections}
On every geometric irreducible component of the fixed-constant
stratum \eqref{eq:root-stratum}, any $m\le r-1$ distinct interior
coefficient functions are algebraically independent. Thus their
joint projection is dominant onto $\A^m$.
\end{lemma}
\begin{proof}
First allow the constant term to vary. The differential of the
root-to-coefficient map is
\begin{equation}\label{eq:tangent-stratum}
 df=h(x)P(x),\qquad
 h(x)=\prod_{\nu=1}^r(x-z_\nu)^{e_\nu-1},\qquad \deg P\le r-1.
\end{equation}
As the roots vary, every such $P$ occurs: the polynomials
$\prod_{\substack{1\le\mu\le r\\\mu\ne\nu}}(x-z_\mu)$ form a basis, and the multipliers
$e_\nu$ are nonzero in characteristic zero.

Take $z_1,\ldots,z_r$ distinct and negative real numbers. Put $k=n-r$.
If $r$ selected nonleading coefficients of a nonzero $df$ vanished,
then $df$, which has degree at most $n-1$, would have at most $k$
nonzero monomials. But $h\mid df$ gives at least $k$ negative real
roots counted with multiplicity. Descartes' rule of signs bounds the
number of negative roots of a polynomial with at most $k$ nonzero
monomials by $k-1$. This is impossible. When $k=0$, vanishing of all
$n$ coefficients is directly impossible for nonzero $df$.
Consequently every $r$-row coefficient Jacobian minor is nonzero
at this point, and any smaller set of coefficient rows is independent.

Now choose the constant-coefficient row together with any $m$ chosen
interior rows. Since $m+1\le r$, the preceding argument says that
these $m+1$ rows are independent at the chosen root tuple. Scale all
roots by a common nonzero complex number $\lambda$. The coefficient
of $x^i$ is multiplied by $\lambda^{n-i}$, while the corresponding
root-coordinate columns are also rescaled by nonzero factors; hence
the rank of this collection of rows is unchanged. We may choose
$\lambda$ so that the constant term becomes $c$.

It remains only to ensure that the rank statement is obtained on
every irreducible component of the fixed-constant root torus. Put
$d_0=\gcd(e_1,\ldots,e_r)$. Its components are distinguished by
the $d_0$ possible values of
$\prod_{\nu=1}^r z_\nu^{e_\nu/d_0}$, whose $d_0$th power is fixed
by the constant term. Once $\lambda^n$ is prescribed in order to
make the constant coefficient equal to $c$, the $n$ choices of
$\lambda$ make $\lambda^{n/d_0}$ run through all $d_0$th roots of
unity. Thus the scaled tuples meet every component. On any such
component, restrict the tangent space to the kernel of the
constant-coefficient differential. Independence of the constant row
and the $m$ interior rows implies that the latter retain rank $m$ on
this kernel. Hence the chosen $m$ coefficient functions are
generically independent on every component, proving dominance.
\end{proof}

\begin{proof}[Proof of Theorem~\ref{thm:instances}(b)]
Let $m=|J|=n-1-s$. On a stratum with $r$ distinct roots, if
$m\le r-1$, Lemma~\ref{lem:coefficient-projections} and the generic
fibre dimension theorem give fibre dimension $r-1-m=s-k$.
If $m>r-1$, the image of the stratum lies in a proper closed subset
after taking its closure, since its dimension is at most $r-1$.
There are only finitely many multiplicity types and components.
Intersect their nonempty open sets of good prescribed values and
remove the closures of their nondominant images. The resulting
nonempty open set can be taken over $\Q$, by intersecting Galois
conjugates. It has the asserted strong index-regularity property.
\end{proof}

\begin{proof}[Proof of Theorem~\ref{thm:instances}(a)]
First suppose $I=\{1,\ldots,s\}$. Prescribing the top
$m=n-1-s$ interior coefficients is equivalent, by Newton's
identities, to prescribing
\[
 \sum_{\nu=1}^r e_\nu z_\nu^j,\qquad 1\le j\le m.
\]
We also prescribe the nonzero product in \eqref{eq:root-stratum}.
After invertible row and column scalings, their Jacobian has rows
\[
 (z_\nu^{-1})_\nu,\ (1)_\nu,\ (z_\nu)_\nu,
                       \ldots,\ (z_\nu^{m-1})_\nu.
\]
Its rank is $\min(r,m+1)$ at every distinct nonzero root tuple,
by the Vandermonde determinant. Every fibre on this stratum
therefore has dimension at most
\[
 \max(r-m-1,0)=\max(s-k,0).
\]
Taking all strata proves index-regularity. The same argument works
in every characteristic greater than $n$ not dividing $c$.

For the final block, use the invertible weighted reversal
\begin{equation}\label{eq:reversal}
 f(x)\longmapsto \frac{x^n}{c}f(c/x)
   =x^n+\sum_{i=1}^{n-1} a_i c^{i-1}x^{n-i}+c^{n-1},
\end{equation}
where in this formula $a_i$ denotes the full interior coefficient,
whether fixed or free. It preserves root multiplicities and
interchanges initial and final free blocks. This proves (a).
\end{proof}

\begin{proof}[Proof of Theorem~\ref{thm:instances}(c)]
There is nothing to prove if $J=\varnothing$. If $|J|=1$, the
prescribed coefficient is nonconstant on every component with
$r\ge2$, by Lemma~\ref{lem:coefficient-projections}. Each of its
level sets has dimension at most $r-2=s-k$. The strata with $r=1$
are finite because their constant term is fixed and nonzero.
This gives \eqref{eq:regular} for every prescribed value.
\end{proof}

\begin{lemma}\label{lem:local-densities}
On an index-regular slice, outside a finite set of primes depending
on the slice,
\begin{equation}\label{eq:density}
 \#\bigl(L(\F_p)\cap\Sigma_k\bigr)
       \ll_n p^{\max(s-k,0)},\qquad 2\le k<n.
\end{equation}
On a strongly index-regular slice the intersection is empty for
$k>s$ outside a finite set of primes.
\end{lemma}
\begin{proof}
Choose equations for the fixed coordinate slice and for the finitely
many subresultant loci $\Sigma_k$, and spread the resulting
scheme-theoretic intersections out over $\Z[1/N]$. Their degrees are
bounded in terms of $n$ (the slice has degree one), independently of
$p$. By upper semicontinuity of fibre dimension, after enlarging $N$
the dimension of every geometric fibre is at most the dimension of
the corresponding characteristic-zero intersection. Thus an
index-regular slice has fibres of dimension at most
$\max(s-k,0)$ for every $p\nmid N$. The standard degree bound for
points on an affine variety of dimension $d$ over $\F_p$ then gives
$O_n(p^d)$ points, proving \eqref{eq:density}.

For a strongly index-regular slice and $k>s$, the characteristic-zero
intersection is empty. After enlarging $N$ once more, its spread-out
model has empty fibres at every $p\nmid N$. Since only finitely many
values of $k$ occur, one common exceptional set of primes suffices.
\end{proof}

\begin{example}\label{ex:nonregular}
Index-regularity does not follow merely from the number of free
coefficients. Take $n=20$, $c=1$, and let $I$ consist of the nine even
interior exponents and the four odd exponents $13,15,17,19$. Set the
other six odd coefficients equal to zero. This is a slice of
dimension $s=13$, satisfying \eqref{eq:threshold}. It contains
\[
 f(x)=h(x^2)^2,\qquad
 h(y)=y^5+u_4y^4+u_3y^3+u_2y^2+u_1y+1.
\]
For generic $\mathbf u$, the polynomial has ten distinct roots,
each of multiplicity two. These polynomials form a four-dimensional
family in $L\cap\Sigma_{10}$, whereas $s-10=3$. In particular a
uniform claim of local density $O(p^{-10})$ on every such slice is
false. This example is not a counterexample to the desired
exceptional-polynomial bound; it is a counterexample to an
unqualified dimension argument for proving that bound.
\end{example}

\section{Uniform archimedean counts and imprimitive groups}\label{sec:imprimitive}

For a degree-$d$ number field $K$ and $u\in K$, put
\[
 \mathcal M_K(u)=\prod_{\sigma:K\hookrightarrow\C}
                         \max(1,|\sigma(u)|).
\]
Complex conjugate embeddings are counted separately. We first record
a uniform box argument, and strengthen its
joint-bound consequence in a form suitable for summing over fields.

\begin{lemma}\label{lem:boxes}
At each archimedean place of a degree-$d$ number field $K$, impose a
bound $|\sigma(u)|\le U_\sigma$, using one modulus bound at a complex
place. Let $d_\sigma=1$ or $2$ according as the place is real or
complex, and put $V=\prod U_\sigma^{d_\sigma}$. The box contains
$O_d(1+V)$ algebraic integers, uniformly in $K$ and the positive
bounds $U_\sigma$.
\end{lemma}
\begin{proof}
Enclose each complex disk in a square and subdivide each of the $d$
real coordinate intervals into $m$ equal pieces. Two distinct
algebraic integers in the same subbox have a nonzero difference $w$
with
\[
 1\le|\Nm_{K/\Q}(w)|\le C_d V/m^d.
\]
Choosing $m$ just larger than $(C_dV)^{1/d}$ makes this impossible.
There are $m^d\ll_d1+V$ subboxes.
\end{proof}

\begin{lemma}\label{lem:joint-box}
Let $v\in\cO_K$, let $R=\mathcal M_K(v)$, and suppose $X\ge R$.
Uniformly in the degree-$d$ field $K$ and in $v$,
\begin{align}
 &\#\left\{u\in\cO_K:
  \prod_{\sigma:K\hookrightarrow\C}
       \max(1,|\sigma(u)|,|\sigma(v)|)\le X\right\}
       \notag\\
 &\hspace{35mm}\ll_d X\bigl(1+\log(X/R)\bigr)^{d-1}.
 \label{eq:joint-box}
\end{align}
\end{lemma}
\begin{proof}
Use the baseline $U_\sigma=\max(1,|\sigma(v)|)$ at each
archimedean place. Decompose $|\sigma(u)|$ dyadically above this
baseline, with an initial range $|\sigma(u)|\le U_\sigma$.
The box indexed by nonnegative integers $k_\sigma$ has volume
$O_d(R2^t)$, where $t=\sum d_\sigma k_\sigma$, and a nonempty
range must satisfy
\[
 t\le\log_2(X/R)+O_d(1).
\]
Lemma~\ref{lem:boxes} bounds its lattice count by $O_d(R2^t)$,
since $R\ge1$. There are $O_d((1+t)^{r_1+r_2-1})$ indices at a
fixed integer $t$. Summing the resulting geometric series, whose
last terms dominate, gives
$O_d(X(1+\log(X/R))^{r_1+r_2-1})$ and hence \eqref{eq:joint-box}.
\end{proof}

For composite $n$, let $\beta(n)$ be as in \eqref{eq:beta-intro}.
If $\ell$ is the smallest prime divisor of $n$, convexity of
$d+n/d$ on the possible divisor interval gives
\begin{equation}\label{eq:beta}
 \beta(n)=\ell+n/\ell-2\le n/2.
\end{equation}

\begin{proposition}\label{prop:imprimitive}
Among all monic degree-$n$ integer polynomials with constant term
$c\ne0$ and height at most $H$, the number that are irreducible and
imprimitive and have no rational proper partial product of roots is
$O_{n,c}(H^{\beta(n)})$.
Consequently, on any slice satisfying \eqref{eq:gcd}, the imprimitive
count is
\begin{equation}\label{eq:imprimitive-slice}
 O_{n,I,\mathbf b,c}\bigl(H^{s-1}+H^{\beta(n)}\bigr).
\end{equation}
For prime $n$ there are no transitive imprimitive groups.
\end{proposition}
\begin{proof}
Let $L_0=\Q(\alpha)$ be the root field of an irreducible
imprimitive polynomial. The block system for the action of the
Galois closure on the conjugates of $\alpha$ is equivalent to the
existence of a proper intermediate field of the root field. Choose
one
\[
 \Q\subsetneq K\subsetneq L_0
\]
minimal among the nontrivial intermediate fields. Thus $K/\Q$ has
no proper intermediate field. Put $d=[K:\Q]$ and $e=[L_0:K]$, so
$de=n$.

Let $g$ be the minimal polynomial of $\alpha$ over $K$:
\[
 g(x)=x^e+u_{e-1}x^{e-1}+\cdots+u_1x+v.
\]
Because $\alpha$ is an algebraic integer, its conjugates over $K$
are algebraic integers; hence the elementary symmetric functions in
those conjugates lie in $\cO_K$, and therefore
$g\in\cO_K[x]$. Every embedding of $K$ into $\C$ extends to $L_0$,
and the roots of the coefficientwise conjugates $\sigma(g)$, as
$\sigma$ ranges over the $d$ embeddings of $K$, are precisely the
$n$ conjugates of $\alpha$ over $\Q$, each once. Consequently
\begin{equation}\label{eq:block-norm}
 f=\Nm_{K/\Q}(g)=\prod_{\sigma:K\hookrightarrow\C}\sigma(g),
 \qquad \Nm_{K/\Q}(v)=c.
\end{equation}

If $v\in\Q$, then $v\in\Z$ and the product of the roots in this
block is $(-1)^e v\in\Q$, contrary to the hypothesis. Hence
$v\notin\Q$. Since $\Q(v)$ is then a nontrivial intermediate field
contained in $K$, the minimality of $K$ gives
\begin{equation}\label{eq:v-generates}
 K=\Q(v).
\end{equation}
This is the reason for choosing a minimal intermediate field rather
than an arbitrary block system.

Let $M$ denote Mahler measure. Multiplicativity and the elementary
coefficient bound for a monic degree-$e$ polynomial give, for every
coefficient $u_j$ of $g$,
\begin{equation}\label{eq:block-height}
 \prod_\sigma\max(1,|\sigma(u_j)|,|\sigma(v)|)
        \le C_n\prod_\sigma M(\sigma(g))
        =C_n M(f)\ll_n H.
\end{equation}
Also $R=\mathcal M_K(v)\le M(f)$. Take $X=C'_nH$ large enough
for all these inequalities.

By \eqref{eq:v-generates}, the minimal polynomial of $v$ has degree
$d$ and the form
\[
 P_v(T)=T^d+c_{d-1}T^{d-1}+\cdots+c_1T+(-1)^dc.
\]
If $\mathcal M_K(v)\le Y$, then each coefficient $c_i$ is an
elementary symmetric function of the $d$ conjugates of $v$, and
hence
\[
 |c_i|\le \binom{d}{i}Y.
\]
The constant coefficient is fixed. There are therefore
$O_{d,c}(Y^{d-1})$ possible polynomials $P_v$; each determines
$K=\Q(v)$ and $v$ up to at most $d$ choices. For each such pair,
Lemma~\ref{lem:joint-box} bounds the choices of the remaining
$e-1$ coefficients by
\begin{equation}\label{eq:fixed-v-count}
 \ll_n X^{e-1}\bigl(1+\log(X/R)\bigr)^{(d-1)(e-1)}.
\end{equation}
The use of independent bounds for these coefficients can only
overcount the polynomials satisfying \eqref{eq:block-height}.

Group the possible $v$ by $R\asymp X2^{-j}$, with the bottom range
truncated at $R=1$. The preceding coefficient count gives at most
$O_{d,c}((X2^{-j})^{d-1})$ possibilities in the $j$th range.
Combining this with \eqref{eq:fixed-v-count} yields
\[
 \ll_{n,c} X^{d+e-2}
       \sum_{j\ge0}2^{-j(d-1)}(1+j)^{(d-1)(e-1)}
       \ll_{n,c} X^{d+e-2}.
\]
The sum converges because $d\ge2$, so in particular it introduces
no power of $\log H$. A fixed $K$ and $g$ determine $f$ uniquely.
Finally sum over the finitely many proper divisors $d$ of $n$, with
$e=n/d$, to obtain $O(H^{\beta(n)})$. Adding
Lemma~\ref{lem:partial-products} gives \eqref{eq:imprimitive-slice}.
\end{proof}

\begin{remark}
The rational partial-product locus cannot simply be omitted when
counting all imprimitive polynomials. For example, block constants
can be rational while the remaining block coefficients generate a
varying field. Proposition~\ref{prop:imprimitive} removes this case
as a hypersurface \emph{on the coefficient slice}, before counting
the remaining fields through their generating block constants.
This reorganizes the block-factor argument
and avoids a separate classification of the nongenerating cases.
\end{remark}

\section{Generators of fixed norm}\label{sec:norm}

We use the following fixed-norm estimate; its short proof is included
for completeness.

\begin{lemma}\label{lem:fixed-norm}
Uniformly over degree-$n$ number fields $K$, for fixed $c\ne0$ and
$X\ge2$,
\begin{equation}\label{eq:fixed-norm}
 \#\{\alpha\in\cO_K:|\Nm_{K/\Q}(\alpha)|=|c|,
                         \mathcal M_K(\alpha)\le X\}
       \ll_{n,c}(1+\log X)^{n-1}.
\end{equation}
Consequently each degree-$n$ field occurs as the root field of at
most $O_{n,c}((1+\log H)^{n-1})$ irreducible polynomials in our family.
\end{lemma}
\begin{proof}
There are $O_{n,c}(1)$ integral ideals of norm $|c|$: only primes
dividing $c$ occur, their exponents are bounded, and there are at
most $n$ prime ideals above any rational prime. The generators of
one principal ideal form a coset of the unit group. Their logarithmic
archimedean absolute values lie in a hyperplane of dimension at most
$n-1$, inside a box of side $O_{n,c}(1+\log X)$.

Modulo roots of unity, logarithmic vectors of units of degree at
most $n$ are uniformly separated. Indeed, Kronecker's theorem and
the finiteness of integral polynomials of bounded degree and bounded
Mahler measure imply that the Mahler measures of nontorsion units
of bounded degree are bounded away from one. For a unit, the sum
of the logarithms is zero, so this gives a lower bound for the
norm of its logarithmic vector. The number of roots of unity of
degree at most $n$ is bounded as well. Packing the log box proves
\eqref{eq:fixed-norm}. Finally a root $\alpha$ of $f$ satisfies
$\Nm(\alpha)=(-1)^n c$ and
$\mathcal M_K(\alpha)=M(f)\le\|f\|_2\ll_n H$.
\end{proof}

\section{Sparse Hermite interpolation and local Fourier bounds}\label{sec:fourier}

Fix a prime $p$ outside a finite set containing the primes dividing
$c$ and all sufficiently small primes in terms of $n$. The free
coefficients identify $L(\F_p)$ with $\F_p^s$. For a function $w$ on
this space, use the normalized transform
\[
 \widehat w(g)=p^{-s}\sum_{a\in\F_p^s}w(a)e_p(g\cdot a),
 \qquad e_p(t)=\exp(2\pi i t/p).
\]
For the zero-frequency estimates put
\begin{equation}\label{eq:rho}
 \rho_k=\min(k,s),\qquad 2\le k<n.
\end{equation}

\begin{lemma}\label{lem:Hermite}
Assume \eqref{eq:threshold}. For distinct nonzero roots $r,t$ over
$\overline{\F}_p$, let $A(r,t)$ be the four-row matrix of the
functionals
\[
 P\longmapsto P(r),\quad P'(r),\quad P(t),\quad P'(t)
\]
on the space spanned by $\{x^i:i\in I\}$. Its rank is at least
three. The rank-three locus in the space of monic squarefree
quadratics $q=(x-r)(x-t)$ has dimension at most one. For each
nonzero fixed $g\in\F_p^s$, the quadratics on which $g$ belongs
to the row space of $A(r,t)$ have dimension at most one in the
rank-four locus, and form a bounded finite set in the rank-three
locus. All degree and cardinality bounds depend only on $n$.
\end{lemma}
\begin{proof}
Under \eqref{eq:threshold} we have $s\ge5$, so the column choices
used below are available. Write $t=rz$, where $r\ne0$ and $z\ne0,1$.
After invertible row
scalings and column scalings by $r^i$, the matrix becomes
\begin{equation}\label{eq:sparse-Hermite}
 B_I(z)=\begin{pmatrix}
 (1)_{i\in I}\\ (i)_{i\in I}\\ (z^i)_{i\in I}\\ (iz^i)_{i\in I}
 \end{pmatrix}.
\end{equation}
For four distinct exponents $i_1<i_2<i_3<i_4$, its determinant is a
nonzero polynomial in $z$. More precisely, its expansion at $z=1$ is
\[
 \det B_{\{i_1,i_2,i_3,i_4\}}(z)
  =\frac{1}{12}\prod_{\mu<\nu}(i_\nu-i_\mu)(z-1)^4
                                      +O((z-1)^5).
\]
To check this, put $z=\exp(u)$, subtract the first two rows from
the last two to their respective orders, and take the first
independent powers $1,i,i^2,i^3$; the coefficient is
$1/4-1/6=1/12$. The displayed coefficient is nonzero for $p>n$
and $p>3$. Thus the rank can fall below four for only $O_n(1)$
ratios $z$.

For a set of at least three distinct exponents, the rank of
\eqref{eq:sparse-Hermite} is two if and only if $z^i$ is constant
on that set. Indeed, rank two would give
\[
 z^i=A+Bi,\qquad iz^i=C+Di.
\]
The quadratic identity
$Bi^2+(A-D)i-C=0$ at three distinct exponents implies
$B=C=0$ and $A=D$. The converse is immediate. If $z\ne1$, constancy
of $z^i$ puts all the exponents in one residue class modulo the
order of $z$, which is at least two. Such a class contains at most
$\lfloor n/2\rfloor$ exponents in $\{1,\ldots,n-1\}$.
By \eqref{eq:threshold}, even after any one column is deleted more
than $\lfloor n/2\rfloor$ columns remain. Hence every one-column
deletion has rank at least three, and so does the original matrix.

Membership of $g$ in the row space of $A(r,rz)$ is equivalent to
membership of $(g_i r^{-i})_i$ in the row space of $B_I(z)$. Choose
$j$ with $g_j\ne0$ and choose four other columns. The corresponding
five-column augmented determinant, expanded in its last row, is a
Laurent polynomial in $r$ with distinct powers $r^{-i}$. The
coefficient of $g_jr^{-j}$ is the nonzero four-column determinant
just discussed. Thus not all membership equations vanish
identically in $(r,z)$. The rank-four membership locus has dimension
at most one.

For each of the boundedly many rank-three ratios $z$, the
one-column deletion assertion allows three columns other than $j$
that are independent, using a suitable choice of three rows. The
four-column augmented determinant using these columns and $j$ is
again a nonzero Laurent polynomial in $r$, because the coefficient
of $g_jr^{-j}$ is nonzero and its power is distinct from the others.
There are only $O_n(1)$ roots $r$. This proves the finite assertion.

These are algebraic rank conditions of bounded degree. Passing
through the finite map $(r,t)\mapsto (-(r+t),rt)$ to the coefficients
of $q$ preserves the dimension bounds. In particular the argument
includes quadratics irreducible over $\F_p$; it does not require
$r,t\in\F_p$.
\end{proof}

\begin{proposition}\label{prop:Fourier}
Let $L$ be index-regular and suppose $n\ge6$ and
\eqref{eq:threshold} holds. There are nonnegative, $O_n(1)$-bounded
weights $W_{k,p}$ majorizing the index-at-least-$k$ indicators such
that, outside a fixed finite set of primes,
\begin{equation}\label{eq:Fourier}
 \widehat W_{k,p}(0)\ll_n p^{-\rho_k},\qquad
 |\widehat W_{k,p}(g)|\ll_n p^{-3}\quad(g\ne0).
\end{equation}
For strongly index-regular slices one may replace $\rho_k$ by $k$;
for $k>s$ the weights can be taken to be zero.
\end{proposition}
\begin{proof}
Let $Z_{k,p}$ be the indicator of index at least $k$. For $k\ge3$,
Lemma~\ref{lem:local-densities} and the trivial Fourier estimate
prove the assertions, since $s\ge5$ and $\rho_k\ge3$.
For index two let $w_3(f)$ count $r\in\F_p^\times$ such that
$(x-r)^3\mid f$, and let $w_{22}(f)$ count monic squarefree
quadratics $q$ with nonzero constant term such that $q^2\mid f$.
An index-exactly-two polynomial either has one triple root or two
double roots. In the first case the triple root is rational over
$\F_p$; in the second the two double roots define a quadratic over
$\F_p$. Hence
\[
 W_{2,p}=w_3+w_{22}+Z_{3,p}
\]
is a bounded nonnegative majorant.

For fixed $r$, the three equations $f(r)=f'(r)=f''(r)=0$ have
rank three in the free coefficients: any three exponent columns
have a nonzero Vandermonde determinant after scalings by powers
of $r$. Their solution set is either empty or an affine subspace
of size $p^{s-3}$. Its character sum vanishes unless $g$ belongs
to the three-row space. For $g\ne0$, choose four exponents including
$j$ with $g_j\ne0$. The augmented determinant has distinct powers
of $r$ in its expansion, and the coefficient from $g_j$ is a
nonzero three-column Vandermonde minor. Only $O_n(1)$ values of
$r$ survive. Thus
\[
 \widehat w_3(0)\ll p^{-2},\qquad
 |\widehat w_3(g)|\ll p^{-3}\quad(g\ne0).
\]

For $w_{22}$ use Lemma~\ref{lem:Hermite}. A quadratic of rank four
has either no solutions or $p^{s-4}$ solutions. There are $O(p^2)$
such quadratics in total and only $O_n(p)$ satisfying the row-space
condition for a fixed nonzero character. The rank-three quadratics
number $O_n(p)$ in total, have affine fibres of size $p^{s-3}$
when consistent, and only $O_n(1)$ survive a fixed nonzero character.
Consequently
\[
 \widehat w_{22}(0)\ll p^{-s}(p^2p^{s-4}+pp^{s-3})\ll p^{-2},
\]
and
\[
 |\widehat w_{22}(g)|\ll p^{-s}(pp^{s-4}+p^{s-3})\ll p^{-3}.
\]
Point counts on the rank and membership loci follow from their
bounded degrees, including for nonsplit quadratics. Combining
these estimates with that for $Z_{3,p}$ proves the proposition.
\end{proof}

We will also need a version in which the number of consecutive free
coefficients, rather than its proportion of the degree, controls
the Fourier argument.

\begin{proposition}\label{prop:block-Fourier}
For an initial free block $I=\{1,\ldots,s\}$ with $s\ge5$,
Proposition~\ref{prop:Fourier} holds without \eqref{eq:threshold}.
The same is true for a final free block.
\end{proposition}
\begin{proof}
Theorem~\ref{thm:instances}(a) supplies the higher-index densities.
For the triple-root weight, after fixing $r$ write
$f=(x-r)^3h$. The prescribed leading block determines the top
coefficients of the monic quotient, and its constant term is fixed
by $c$. Exactly $s-3$ quotient coefficients remain free. Their
variations in $f$ are
\[
 x^j(x-r)^3,\qquad 1\le j\le s-3.
\]
For $g\ne0$, at least one annihilation condition is a nonzero
polynomial in $r$: if all vanished identically, expansion in $r$
would show that the functional associated to $g$ annihilates every
monomial $x,\ldots,x^s$. Hence only $O_n(1)$ roots survive and the
normalized transform is $O(p^{-3})$.

Similarly, for $f=q^2h$, with $q=x^2+ux+v$ squarefree and $v\ne0$,
there are $s-4\ge1$ free quotient coefficients, with variations
\[
 x^j(x^2+ux+v)^2,\qquad 1\le j\le s-4.
\]
At least one annihilation condition is a nonzero polynomial in
$(u,v)$. Otherwise specialize $q=(x-r)^2$ and expand in $r$ to
obtain the same contradiction. Thus $O_n(p)$ quadratics survive,
giving $p\,p^{s-4}/p^s=O(p^{-3})$. The zero-frequency bounds are
$O(p^{-2})$ by the same parameter counts. The weights $Z_{k,p}$
for $k\ge3$ are as before. Weighted reversal
\eqref{eq:reversal} proves the final-block assertion.
\end{proof}

\section{Poisson summation and the discriminant tail}\label{sec:poisson}

It is convenient to separate the analytic argument from the precise
Fourier gain. Suppose $0<\theta\le1$ and the local weights satisfy
\begin{equation}\label{eq:abstract-Fourier}
 \widehat W_{k,p}(0)\ll p^{-\rho_k},\qquad
 |\widehat W_{k,p}(g)|\ll p^{-2-\theta}\quad(g\ne0),
\end{equation}
where
\begin{equation}\label{eq:rho-condition}
 \rho_k-\frac{k}{2}\ge1\qquad(2\le k<n).
\end{equation}
The previous section gives $\theta=1$. For the capped values
$\rho_k=\min(k,s)$, condition \eqref{eq:rho-condition} holds whenever
$s\ge(n+1)/2$, in particular under \eqref{eq:threshold}. For strongly
index-regular slices, $\rho_k=k$ satisfies it in every dimension.

\begin{lemma}\label{lem:Poisson}
Assume $s\ge3$ and \eqref{eq:abstract-Fourier}--\eqref{eq:rho-condition}.
Let $C$ be squarefree and supported on good primes, let
$2\le k_p<n$ for $p\mid C$, and put $D=\prod_{p\mid C}p^{k_p}$.
The number of polynomials in $\cF(H)$ with index at least $k_p$
modulo every $p\mid C$ is
\begin{equation}\label{eq:individual-Poisson}
 \ll K_n^{\omega(C)}H^s\prod_{p\mid C}p^{-\rho_{k_p}}
                    +C^{s-2-\theta+\epsilon}.
\end{equation}
The union of these conditions over $C\le X$ and $D>Y$ has size
\begin{equation}\label{eq:aggregate}
 \ll_{\epsilon} H^sY^{-1/2+\epsilon}
                         +X^{s-1-\theta+\epsilon}.
\end{equation}
All constants may depend on the fixed slice and on $\theta$.
\end{lemma}
\begin{proof}
Choose a nonnegative Schwartz majorant of the coefficient box with
compactly supported Fourier transform, and apply Poisson summation
to the product of the local weights. The zero-frequency contribution
is the first term of \eqref{eq:individual-Poisson}. By the Chinese
remainder theorem, the coefficient at a nonzero integer dual vector
$g=(g_1,\ldots,g_s)$ has absolute value at most
\[
 K_n^{\omega(C)}C^{-2-\theta}
                         \gcd(C,g_1,\ldots,g_s)^\theta.
\]
Here a prime dividing every coordinate of $g$ contributes at most
$p^{-2}$, and every other prime contributes at most
$p^{-2-\theta}$. Only $\|g\|_\infty\ll C/H=:R$ occur.
If this range contains nonzero integer vectors, then
\begin{align*}
 \sum_{0<\|g\|_\infty\ll R}\gcd(C,g_1,\ldots,g_s)^\theta
 &\ll R^s\sum_{q\mid C}q^{\theta-s}\\
 &\ll_{s,\theta}R^s,
\end{align*}
since $s\ge3$ and $\theta\le1$. Multiplying by $H^sC^{-2-\theta}$
gives the claimed error; $K_n^{\omega(C)}$ is absorbed in $C^\epsilon$.
If the nonzero dual range is empty, there is no error.

For the main terms in the union bound, set
$\alpha=1/2-\epsilon$ with $0<\epsilon<1/2$. The weighted tail
satisfies
\begin{align}
 &\sum_{\substack{D>Y\\2\le v_p(D)\le n-1}}
       K_n^{\omega(D)}\prod_{p\mid D}p^{-\rho_{v_p(D)}}
       \notag\\
 &\qquad\le Y^{-\alpha}
       \prod_p\left(1+K_n\sum_{k=2}^{n-1}
                                   p^{\alpha k-\rho_k}\right)
       \ll_{n,\epsilon}Y^{-1/2+\epsilon}.
       \label{eq:weighted-tail}
\end{align}
The Euler product converges because
$\rho_k-\alpha k\ge1+k\epsilon$ by
\eqref{eq:rho-condition}. For each $C$ there are at most
$(n-2)^{\omega(C)}\ll_{n,\epsilon}C^\epsilon$ exponent patterns.
Summing the errors over $C\le X$ proves \eqref{eq:aggregate}, after
rescaling $\epsilon$.
\end{proof}

\begin{remark}\label{rem:capped-density}
The weight in \eqref{eq:weighted-tail} is essential. If $k>s$, a
zero-dimensional intersection need not be empty, so its normalized
mass is generally $O(p^{-s})$, not $O(p^{-k})$. The inequality
$\rho_k-k/2\ge1$ is precisely what permits this loss without
spoiling the squarefull tail. Strong index-regularity is a different,
stronger way to avoid the issue.
\end{remark}

\section{Elimination in an arbitrary free coefficient}\label{sec:elimination}

Fix any $j\in I$, write $T=a_j$, and let $\mathbf a'$ denote the
other $s-1$ free coefficients. Put
\[
 \Delta(\mathbf a',T)=\disc_x(f),\qquad
 \mathcal R_j(\mathbf a')=\Res_T(\Delta,\partial_T\Delta).
\]
The large-prime argument requires that this resultant not vanish
identically and that the degree of $\Delta$ in $T$ not drop modulo
large primes. Both properties hold for every free position.

\begin{lemma}\label{lem:discriminant}
As a polynomial in $T$, the discriminant has degree $n$ and constant
leading coefficient of absolute value
\begin{equation}\label{eq:disc-leading}
 j^j(n-j)^{n-j}|c|^{j-1}.
\end{equation}
If $\dim(L\cap\Sigma_2)\le s-2$, then $\mathcal R_j$ is a nonzero polynomial
in $\mathbf a'$. Its degree is bounded in terms of $n$.
\end{lemma}
\begin{proof}
Consider the roots as formal Puiseux series as $T\to\infty$.
There are $j$ small roots of size $T^{-1/j}$, whose leading scaled
values satisfy $y^j+c=0$, and $n-j$ large roots of size
$T^{1/(n-j)}$, whose leading scaled values satisfy $y^{n-j}+1=0$.
All other coefficients are held fixed in this limit. In the product
of squared root differences, the small-small pairs contribute
$T^{-(j-1)}$, the large-large pairs contribute $T^{n-j-1}$,
and the cross pairs contribute $T^{2j}$. The total exponent is $n$.
The two binomial discriminants and the cross product give the
nonzero coefficient \eqref{eq:disc-leading}, up to sign. This also
proves that it is independent of the other coefficients.

At a polynomial with exactly one double root $r$ and all other
roots simple, write $f=(x-r)^2h$ with $h(r)\ne0$. The first-order
discriminant identity for the perturbation $f+\varepsilon x^j$ is
\begin{equation}\label{eq:disc-derivative}
 \partial_T\Delta=-4r^j h(r)^3\disc(h)\ne0.
\end{equation}
For example, the two nearby roots differ to first order by the
square root of $-4\varepsilon r^j/h(r)$, which gives the displayed
identity upon taking the remaining root differences. The root
$r$ is nonzero because $c\ne0$. At a triple root or at two double
roots, the discriminant and all coefficient derivatives vanish.
Thus the simultaneous zero set of $\Delta$ and $\partial_T\Delta$
is exactly $L\cap\Sigma_2$ in characteristic zero. If $\mathcal R_j$
were identically zero, the two polynomials would have a common
factor over $\Q(\mathbf a')[T]$, and their common zero locus would
have a component of dimension $s-1$. This contradicts the assumed
dimension bound. The degree assertion follows from the bounded
degrees of discriminants and resultants.
\end{proof}

Let $K_f$ be the root field of an irreducible polynomial with proper
primitive Galois group. Choose a finite set $S$ of exceptional primes
containing all those already excluded, and put
\begin{equation}\label{eq:D0C0}
 D_0=\prod_{p\notin S}p^{v_p(|\disc K_f|)},\qquad C_0=\rad(D_0).
\end{equation}
For fixed degree, the discriminant exponents at each fixed rational
prime are bounded, so $|\disc K_f|\ll_{n,S}D_0$. At a tame
ramified prime $p\notin S$, let $k_p=v_p(|\disc K_f|)$. Tame
inertia identifies $k_p$ with the index of an inertia generator.
A proper primitive group contains no transposition. Consequently
\begin{equation}\label{eq:tame-indices}
 2\le k_p\le n-1,
 \qquad f\bmod p\in\Sigma_{k_p},
 \qquad D_0\ge C_0^2.
\end{equation}
For the middle assertion we use the standard index comparison at a
tame prime,
\[
 \ind(f\bmod p)\ge v_p(|\disc K_f|),
\]
as in \cite[Lemma~24]{Bhargava}. This formulation is important
because the order $\Z[\alpha]$ need not be $p$-maximal, so a naive
factorization of $f\bmod p$ need not record the prime factorization
of $p\cO_{K_f}$. In the $p$-maximal case the inequality reduces to
the familiar identity
$v_p(\disc K_f)=n-\sum_{\mathfrak p\mid p}f_{\mathfrak p}$
together with the corresponding factorization count.

Since index at least two forces both $\Delta$ and its coefficient
derivatives to vanish modulo $p$, we obtain
\begin{equation}\label{eq:rad-resultant}
 C_0\mid \mathcal R_j(\mathbf a').
\end{equation}

\begin{lemma}\label{lem:large-primes}
Suppose $\dim(L\cap\Sigma_2)\le s-2$. Fix $\eta>0$ and put
\[
 A_0=\prod_{\substack{p\mid C_0\\p>H^\eta}}p.
\]
The number of these proper primitive polynomials for which $A_0>H$
is $O_{n,I,\mathbf b,c,\eta}(H^{s-1})$.
\end{lemma}
\begin{proof}
By Lemma~\ref{lem:discriminant}, the tuples with $\mathcal R_j(\mathbf a')=0$
number $O(H^{s-2})$, and allowing all $T$ costs $O(H^{s-1})$.
For each other tuple, $0<|\mathcal R_j(\mathbf a')|\ll H^{L_n}$, where
$L_n$ depends only on $n$. It has only boundedly many prime divisors
exceeding $H^\eta$, and hence only boundedly many possible products
$A_0$. For any such product, the congruence
\[
 \Delta(\mathbf a',T)\equiv0\pmod{A_0}
\]
has at most $n^{\omega(A_0)}=O_{n,\eta}(1)$ residue classes: the
leading coefficient \eqref{eq:disc-leading} is a unit at every
prime under consideration. Because $A_0>H$, each class has $O(1)$
representatives in $[-H,H]$. Sum over the $O(H^{s-1})$ choices of
$\mathbf a'$.
\end{proof}

\section{The primitive-group sieve and proof of the main theorem}\label{sec:completion}

Let $F_n(G,X)$ denote the number of isomorphism classes of degree-$n$
fields of absolute discriminant at most $X$ whose Galois closure acts
on the $n$ embeddings as the transitive group $G\le S_n$. We state
a transfer result so that improved field counts can be inserted
without repeating the sieve.

\begin{proposition}\label{prop:transfer}
Let $L$ be a fixed slice \eqref{eq:family}, with $s\ge3$ and
$\dim(L\cap\Sigma_2)\le s-2$. Suppose it admits bounded
nonnegative local majorants satisfying
\eqref{eq:abstract-Fourier}--\eqref{eq:rho-condition} for some
$0<\theta\le1$. Let $G<S_n$ be proper primitive and suppose
\begin{equation}\label{eq:field-exponent}
 F_n(G,X)\ll_{n,G,\epsilon}X^{a(G)+\epsilon}.
\end{equation}
If
\begin{equation}\label{eq:transfer-threshold}
 2a(G)<s-1,
\end{equation}
then there are $O(H^{s-1})$ polynomials in $\cF(H)$ with Galois
group $G$, with the implied constant depending on the fixed family
and on $G$.
\end{proposition}
\begin{proof}
Choose $\delta>0$ sufficiently small that
\begin{equation}\label{eq:delta-choice}
 2a(G)(1+\delta)<s-1,
 \qquad (s-1-\theta)(1+\delta)<s-1.
\end{equation}
Let $D_0,C_0$ be as in \eqref{eq:D0C0}. If
$D_0\le H^{2+2\delta}$, the field count and
Lemma~\ref{lem:fixed-norm} give
\[
 \ll H^{2a(G)(1+\delta)+\epsilon}(1+\log H)^{n-1}
      \ll H^{s-1},
\]
where $\epsilon$ is chosen smaller than the fixed margin in
\eqref{eq:delta-choice}.

If $D_0>H^{2+2\delta}$ and $C_0\le H^{1+\delta}$, apply
Lemma~\ref{lem:Poisson} with $X=H^{1+\delta}$ and
$Y=H^{2+2\delta}$. The main exponent is
$s-1-\delta+O(\epsilon)$; the error exponent is
$(s-1-\theta)(1+\delta)+O(\epsilon)$. Both are strictly less
than $s-1$ for a sufficiently small $\epsilon$.

Finally suppose $C_0>H^{1+\delta}$, and put $\eta=\delta/4$.
The case $A_0>H$ is handled by Lemma~\ref{lem:large-primes}.
Otherwise begin with $A_0$ and multiply successively by prime
factors of $C_0/A_0$, each at most $H^{\delta/4}$, until the
resulting divisor $B_0\mid C_0$ lies in
\begin{equation}\label{eq:middle-band}
 H^{1+3\delta/4}<B_0\le H^{1+\delta}.
\end{equation}
This is possible because $C_0>H^{1+\delta}$. The corresponding
$D'=\prod_{p\mid B_0}p^{k_p}$ exceeds $H^{2+3\delta/2}$, by
\eqref{eq:tame-indices}. Apply Lemma~\ref{lem:Poisson} to all
possible $B_0,D'$ with $X=H^{1+\delta}$ and
$Y=H^{2+3\delta/2}$. Its main exponent is now
$s-1-3\delta/4+O(\epsilon)$, while its error exponent is unchanged.
Again both are strictly less than $s-1$. These three cases exhaust
the polynomials and prove the proposition.
\end{proof}

Bhargava's primitive field-counting theorem
\cite[Theorem~20]{Bhargava} states that
\begin{equation}\label{eq:Bhargava-fields}
 F_n(G,X)\ll_{n,\epsilon}
 X^{(n+2)/4-1+1/\ind(G)+\epsilon},
\end{equation}
where $\ind(G)$ is the least nonzero index of an element of $G$.
For a proper primitive group, $\ind(G)\ge2$, and hence we may take
\begin{equation}\label{eq:a-n4}
 a(G)=n/4.
\end{equation}

\begin{proof}[Proof of Theorem~\ref{thm:main}]
The lower bound is Lemma~\ref{lem:lower}. For the upper bound,
\eqref{eq:gcd} holds by Remark~\ref{rem:gcd-dense}, so
Lemma~\ref{lem:partial-products} counts all reducibles and all
rational partial-product cases in $O(H^{s-1})$. The discriminant
is a nonzero polynomial on the slice, by
Lemma~\ref{lem:discriminant}, so discriminant-zero polynomials also
number $O(H^{s-1})$.

The remaining imprimitive polynomials contribute
$O(H^{\beta(n)})$ by Proposition~\ref{prop:imprimitive}, with
$\beta(n)\le n/2<s-1$; in prime degree there are none. For proper
primitive groups, Proposition~\ref{prop:Fourier} supplies
\eqref{eq:abstract-Fourier} with $\theta=1$ and the capped
$\rho_k$ in \eqref{eq:rho}. These satisfy
\eqref{eq:rho-condition}. Combining \eqref{eq:a-n4} with
\eqref{eq:threshold} gives $2a(G)<s-1$, so
Proposition~\ref{prop:transfer} applies. There are only finitely
many permutation groups of degree $n$. Summing their contributions
proves the upper bound in \eqref{eq:main}.
\end{proof}

\begin{proof}[Proof of Corollary~\ref{cor:one-extra}]
Here $s=n-2$, so $s>n/2+1$ is equivalent to $n>6$.
Theorem~\ref{thm:instances}(c) supplies index-regularity for every
position and value of the additional coefficient.
\end{proof}

\section{Degrees and endpoint cases}\label{sec:edges}

The smallest number of varying coefficients in the main theorem,
and the corresponding number of additional prescribed interior
coefficients, are as follows. The leading coefficient and the
nonzero constant term are not included in the last column.
\begin{center}
\begin{tabular}{@{}rrrr@{}}
\toprule
Degree $n$ & Free coefficients $s$ & Exceptional count & Additional fixed\\
\midrule
6 & 5 & $\asymp H^4$ & 0\\
7 & 5 & $\asymp H^4$ & 1\\
8 & 6 & $\asymp H^5$ & 1\\
9 & 6 & $\asymp H^5$ & 2\\
10 & 7 & $\asymp H^6$ & 2\\
11 & 7 & $\asymp H^6$ & 3\\
12 & 8 & $\asymp H^7$ & 3\\
13 & 8 & $\asymp H^7$ & 4\\
\bottomrule
\end{tabular}
\end{center}
In degrees seven and eight, the additional coefficient may be at
any position and have any integral value. In the rows with two or
more additional fixed coefficients, arbitrary positions require the
genericity qualification of Theorem~\ref{thm:instances}(b), unless
index-regularity is verified separately. Initial and final free
blocks require no qualification on the prescribed values.

\subsection{Even degrees: isolating the alternating group}
Let $n$ be even and consider the next smaller value
\[
 s=n/2+1.
\]
The imprimitive estimate already gives the target $O(H^{n/2})$,
including its endpoint, because Proposition~\ref{prop:imprimitive}
has no logarithmic factor. For an initial or final free block and
$n\ge8$, the local Fourier argument also works: $s\ge5$, and
$\rho_k=\min(k,s)$ satisfies \eqref{eq:rho-condition}. The remaining
issue for the primitive sieve is that \eqref{eq:a-n4} gives equality
rather than strict inequality in \eqref{eq:transfer-threshold}.

For $n\ge9$, a proper primitive group other than $A_n$ has
$\ind(G)\ge3$; see the primitive-group index bounds in
\cite[Section~2]{Bhargava}. Thus \eqref{eq:Bhargava-fields} gives
\[
 a(G)\le n/4-1/6<n/4\qquad(G\ne A_n).
\]
We therefore obtain a further unconditional statement.

\begin{corollary}\label{cor:even-edge}
Let $n\ge10$ be even, $c\ne0$, and let the free coefficients form
an initial or final block of size $s=n/2+1$. For every choice of
the remaining fixed coefficients, the number of exceptional
polynomials \emph{other than irreducible polynomials with Galois
group $A_n$} is $O(H^{n/2})$.
\end{corollary}
\begin{proof}
Use Proposition~\ref{prop:block-Fourier}, the preceding strict field
exponents, and Proposition~\ref{prop:transfer}. Reducible,
discriminant-zero, and imprimitive polynomials are already counted
in $O(H^{n/2})$ by the preceding arguments.
\end{proof}

Consequently any field-counting input $F_n(A_n,X)\ll X^{a+\epsilon}$
with $a<n/4$ completes this even-degree endpoint for these block
families. This is a criterion in terms of the field estimate used,
not an assertion that $n/4$ is the best available exponent in every
degree. It also shows exactly where future improvements can enter.
In degree eight, the proper primitive group $8T48$ has an improved
field-counting exponent strictly below $2$ by
\cite[Theorem~1.8 and Table~2]{ChowDietmann}; this removes that
particular small-degree obstruction to the same comparison.

\subsection{Degrees three, four, five, and six}
For $n=3$ with only the leading and constant terms fixed, $s=2$
and the target count is $O(H)$. Index at least two modulo a good
prime forces
\[
 f(x)=(x-r)^3,\qquad -r^3=c.
\]
This is a zero-dimensional locus. A nonzero normalized Fourier
coefficient is in general only $O(p^{-2})$, with no positive
$\theta$ in \eqref{eq:abstract-Fourier}. The aggregate error at
$\theta=0$ would be $X$, while the middle-band argument needs a
cutoff beyond $H$. In addition, even the expected cyclic-cubic
field exponent $a=1/2$ gives $2a=s-1$, rather than a strict
inequality. Thus the present method leaves a logarithmic gap in the
cyclic cubic case. These are limitations of the proof, not evidence
against the desired cubic bound.

We now complete the proof of Theorem~\ref{thm:fixed-constant} in the
two degrees not covered by Theorem~\ref{thm:main}.

\begin{proposition}\label{prop:quintic-fixed}
For fixed $c\ne0$,
\[
 E_{5,c}(H)\ll_c H^3.
\]
\end{proposition}
\begin{proof}
The reducible and discriminant-zero contributions are $O_c(H^3)$ by
Lemma~\ref{lem:partial-products} and Lemma~\ref{lem:discriminant}.
There are no transitive imprimitive groups of prime degree, so it
remains to treat proper primitive groups.

The fixed-constant slice has $s=4$, and it is index-regular by
Theorem~\ref{thm:instances}(a). For index at least three, the
trivial Fourier estimate together with Lemma~\ref{lem:local-densities} gives
$O(p^{-3})$. For index two, use the same majorant
$W_{2,p}=w_3+w_{22}+Z_{3,p}$ as in the proof of
Proposition~\ref{prop:Fourier}. The triple-root term has
$|\widehat w_3(g)|\ll p^{-3}$ for $g\ne0$. The only missing term is
$w_{22}$, because in degree five the quotient by the square of a
quadratic has no free coefficient.

Write
\[
 f=(x^2+ux+v)^2(x+c/v^2),\qquad v\in\F_p^\times.
\]
If $g=(g_1,g_2,g_3,g_4)$ corresponds to the coefficients of
$x^4,x^3,x^2,x$, the phase as a polynomial in $u$ has quadratic and
linear coefficients
\[
 g_2+\frac{cg_3}{v^2},\qquad
 2g_1+\frac{2cg_2}{v^2}+2vg_3+\frac{2cg_4}{v}.
\]
For $g\ne0$ these cannot both vanish identically as functions of $v$:
the first would give $g_2=g_3=0$, and then the second gives
$g_1=g_4=0$. Hence only $O(1)$ values of $v$ have constant phase.
For all other $v$ the quadratic Gauss-sum bound gives $O(p^{1/2})$
for the $u$-sum. Summing over $v$ and normalizing by $p^4$ yields
\[
 |\widehat w_{22}(g)|\ll p^{-5/2}.
\]
The zero-frequency mass is $O(p^{-2})$. Thus
\eqref{eq:abstract-Fourier} holds with $\theta=1/2$ and
$\rho_k=\min(k,4)$; condition \eqref{eq:rho-condition} is satisfied.
Bhargava's field bound gives $a(G)\le5/4$, so
$2a(G)\le5/2<3=s-1$. Proposition~\ref{prop:transfer} therefore
counts every proper primitive group in $O_c(H^3)$. There are only
finitely many such groups, proving the result.
\end{proof}

The quartic endpoint has a different feature. Put
\[
 f=x^4+Ax^3+Bx^2+Dx+c.
\]
For a good prime, an index-at-least-two quartic is covered by the
triple-root family
\[
 T_r=(x-r)^3(x-c/r^3),\qquad r\in\F_p^\times,
\]
or a square family
\[
 S_{v,u}=(x^2+ux+v)^2,\qquad v^2=c.
\]
The square family has exceptional Fourier directions when $c$ is a
square. The following estimate treats them explicitly.

\begin{lemma}\label{lem:quartic-middle}
Fix $c\ne0$ and $2\le Q\le X$. The number of quartics in the
fixed-constant box that have index at least two modulo every prime
dividing some squarefree $q\in(Q,X]$, supported outside a fixed
finite set of exceptional primes, is
\begin{equation}\label{eq:quartic-middle}
 \ll_{c,\epsilon}H^2+H^3Q^{-1+\epsilon}
              +H^{-1}X^{5/2+\epsilon}
              +H^{5/2}Q^{-1/2+\epsilon}.
\end{equation}
The last term is needed only when $c$ is a positive integer square.
\end{lemma}
\begin{proof}
Use the representation counts of $T_r$ and $S_{v,u}$ as nonnegative
local weights. Their zero Fourier coefficients are $O(p^{-2})$.
For $g=(g_1,g_2,g_3)\ne0$, the triple-root phase is a nonconstant
Laurent polynomial in $r$: its positive powers $r,r^2,r^3$ have
coefficients $-3g_1,3g_2,-g_3$. The one-variable Weil bound therefore
gives $O_c(p^{-5/2})$ after normalization by $p^3$.

For the square family the transform is
\begin{equation}\label{eq:square-quartic-fourier}
 p^{-3}e_p(2vg_2)\sum_{u\bmod p}
 e_p\bigl(g_2u^2+2(g_1+vg_3)u\bigr).
\end{equation}
It is $O(p^{-5/2})$ unless
\[
 g_2=0,\qquad g_1+vg_3=0;
\]
it is zero when the first equality holds and the second does not, and
is $p^{-2}$ when both hold.

Apply compactly supported Fourier smoothing as in
Lemma~\ref{lem:Poisson}. For modulus $q$, the integer dual vectors
satisfy $\|g\|_\infty\ll q/H=:R$. First remove the rational cone
\[
 g_2=0,\qquad g_1^2-cg_3^2=0.
\]
Outside it, the Fourier coefficient of the product weight is bounded
by
\[
 K^{\omega(q)}q^{-5/2}
 \gcd(q,g_2,g_1^2-cg_3^2)^{1/2}.
\]
The gcd is $O_c(R^2)$, and there are $O(R^3)$ dual vectors when
$R\ge1$. Their Poisson contribution is therefore
\[
 \ll_{c,\epsilon}H^3q^{-5/2+\epsilon}R^4
 \ll_{c,\epsilon}q^{3/2+\epsilon}/H.
\]
If $c$ is not a positive rational square, there are no nonzero
rational points on the removed cone.

Now suppose $c=t^2$ with $t\in\Z_{>0}$. Remove the two hyperplanes
$D=\pm tA$, which contain $O_c(H^2)$ quartics. Expand the product of
local weights by assigning at each prime one of the types
$T,S_+,S_-$. Let $q_+$ and $q_-$ be the products of primes assigned
$S_+$ and $S_-$ respectively. Off the removed hyperplanes,
\[
 q_+\mid D-tA,\qquad q_-\mid D+tA,
\]
so $q_+,q_-\ll_c H$. On the cone $g=(-th,0,h)$, the $S_-$ factor
vanishes unless its prime divides $h$; in all cases the product
transform is bounded by
\[
 K^{\omega(q)}q^{-5/2}q_+^{1/2}
             \gcd(h,q/q_+)^{1/2}.
\]
Since
\[
 \sum_{1\le|h|\le R}\gcd(h,q/q_+)^{1/2}
       \ll_\epsilon Rq^\epsilon,
\]
and $q_+\ll_cH$, its Poisson contribution is
$O_{c,\epsilon}(H^{5/2}q^{-3/2+\epsilon})$. The other cone is
identical. The number of assignments is $3^{\omega(q)}\ll_\epsilon
q^\epsilon$.

The zero-frequency contribution is
$O_\epsilon(H^3q^{-2+\epsilon})$. Summing these estimates over
$Q<q\le X$ and adding the removed hyperplanes gives
\eqref{eq:quartic-middle}.
\end{proof}

\begin{proposition}\label{prop:quartic-fixed}
For fixed $c\ne0$,
\[
 E_{4,c}(H)\ll_c H^2.
\]
\end{proposition}
\begin{proof}
The fixed-constant quartic slice is index-regular by
Theorem~\ref{thm:instances}(a). Reducible and discriminant-zero
quartics contribute $O_c(H^2)$ by Lemmas~\ref{lem:partial-products}
and \ref{lem:discriminant}.
Proposition~\ref{prop:imprimitive} gives $O_c(H^2)$ for the
irreducible imprimitive quartics, including the rational
partial-product cases already removed by Lemma \ref{lem:partial-products}.
Thus only the primitive group $A_4$ remains.

For an $A_4$-field, every nontrivial tame inertia permutation has
index two. Hence, in the notation \eqref{eq:D0C0},
$D_0=C_0^2$. The field-counting bound
\cite[Theorem~1.4]{BSTTTZ} implies, for example,
\[
 F_4(A_4,Z)\ll Z^{4/5}.
\]
Fix $\delta=1/100$. If $C_0\le H^{1+\delta}$, then
$D_0\le H^{2+2\delta}$ and Lemma~\ref{lem:fixed-norm} gives
\[
 O_c\bigl(H^{(8/5)(1+\delta)}(\log H)^3\bigr)=O_c(H^2).
\]

Suppose now that $C_0>H^{1+\delta}$, and put $\eta=\delta/4$.
If the product $A_0$ of prime divisors of $C_0$ exceeding $H^\eta$
is greater than $H$, Lemma~\ref{lem:large-primes} applies. Otherwise,
multiply $A_0$ successively by prime factors of $C_0/A_0$ until a
divisor $B_0\mid C_0$ satisfies
\[
 H^{1+3\delta/4}<B_0\le H^{1+\delta}.
\]
Apply Lemma~\ref{lem:quartic-middle} with
\[
 Q=H^{1+3\delta/4},\qquad X=H^{1+\delta}.
\]
Apart from the term $H^2$, the three exponents in
\eqref{eq:quartic-middle}, before arbitrarily small $\epsilon$-losses,
are
\[
 2-3\delta/4,\qquad 3/2+5\delta/2,
 \qquad 2-3\delta/8,
\]
all strictly below two. This proves the proposition.
\end{proof}

\begin{proof}[Proof of Theorem~\ref{thm:fixed-constant}]
The lower bound is Lemma~\ref{lem:lower} with $s=n-1$.
For $n\ge6$, take $I=\{1,\ldots,n-1\}$ and $J=\varnothing$ in
Theorem~\ref{thm:main}; Theorem~\ref{thm:instances}(a) supplies
index-regularity and the inequality $n-1>n/2+1$ holds. The cases
$n=5$ and $n=4$ are Propositions~\ref{prop:quintic-fixed} and
\ref{prop:quartic-fixed}, respectively.
\end{proof}

Degree six with one additional interior coefficient fixed, so
$s=4$, is a particularly concrete next test. The slice is
index-regular for every position and value, and
Proposition~\ref{prop:imprimitive} already bounds all its imprimitive
polynomials by $O(H^3)$. Two independent endpoint issues remain.
First, Proposition~\ref{prop:Fourier} does not cover $s=4$, so one
needs a four-variable analysis of the index-two locus. Second, for
$A_6$ the general field bound has $a(A_6)=3/2$, giving
$2a(A_6)=3=s-1$ rather than the strict inequality required in
Proposition~\ref{prop:transfer}. Thus even after the local Fourier
problem is solved, the alternating group needs a sharper field count
(or a separate argument).

\section{What Malle's conjecture would improve}\label{sec:Malle}

Only the upper-bound form of Malle's conjecture is relevant here:
for each transitive permutation group $G\le S_n$,
\begin{equation}\label{eq:Malle}
 F_n(G,X)\ll_{n,G,\epsilon}X^{1/\ind(G)+\epsilon}.
\end{equation}
We do not use a conjectural leading constant or a prescribed power
of $\log X$; see \cite{Malle,ChowDietmann} for the field-counting
conjecture and its formulation. For proper primitive groups,
\eqref{eq:Malle} replaces $a(G)=n/4$ by $a(G)\le1/2$.
The field-counting inequality in Proposition~\ref{prop:transfer}
then becomes simply
\begin{equation}\label{eq:Malle-threshold}
 s \geq 3.
\end{equation}
This removes the half-dimensional restriction from the
\emph{primitive field-counting step}. It does not by itself prove
the required Fourier estimates or the high-index density bounds
on every three-dimensional slice.

There is nevertheless a concrete conditional theorem using only
the local results proved in this paper.

\begin{corollary}\label{cor:Malle}
Assume \eqref{eq:Malle} for the proper primitive groups of degree
$n$. Let $5\le s\le n-1$, let the free coefficients be an initial
block, and choose the remaining fixed coefficients in the nonempty
open set $U_{I,c}$ of Theorem~\ref{thm:instances}(b). Then the
proper primitive exceptional count is $O(H^{s-1})$.
Moreover,
\[
 E_{\cF}(H)\asymp H^{s-1}
\]
if $n$ is prime, or if $n$ is composite and
\begin{equation}\label{eq:Malle-imprimitive}
 s\ge\beta(n)+1=\ell+n/\ell-1,
\end{equation}
where $\ell$ is the smallest prime divisor of $n$.
The analogous assertion holds for a final free block with generic
prescribed coefficients.
\end{corollary}
\begin{proof}
The genericity assumption is used here for strong, rather than merely
ordinary, index-regularity: Theorem~\ref{thm:instances}(b) gives
$\rho_k=k$, including zero weights for $k>s$.
Proposition~\ref{prop:block-Fourier} supplies
$\theta=1$ for every $s\ge5$, independently of $n$.
Proposition~\ref{prop:transfer} now applies by
\eqref{eq:Malle-threshold}. An initial block contains exponent one,
so Lemma~\ref{lem:partial-products} counts its rational
partial-product cases. Proposition~\ref{prop:imprimitive} supplies
the remaining count under \eqref{eq:Malle-imprimitive}; there is
no imprimitive contribution in prime degree. The discriminant-zero
count and the lower bound are unchanged. Weighted reversal proves
the final-block version.
\end{proof}

Thus, in prime degrees $n\ge7$, the conditional theorem allows
only five varying coefficients, regardless of $n$. For odd
composite degrees with smallest prime divisor three it allows
$s\ge n/3+2$, and for $n=\ell^2$ with $\ell$ an odd prime it
allows $s\ge2\ell-1$. For even degrees the imprimitive estimate
used here still asks for $s\ge n/2+1$. These are sufficient ranges,
not proposed optimal thresholds.

Finally, the formal inequality \eqref{eq:Malle-threshold} suggests
that three varying parameters could suffice for the primitive
part when suitable local cancellation and density statements are
available. Our degree-independent block Fourier proof requires
five parameters, and we do not assert an arbitrary three-parameter
theorem. At two parameters, the absence of a positive Fourier gain
and the equality in the field-counting comparison remain even
under \eqref{eq:Malle}. This separates the limitations of present
field counts, the local geometry of coefficient slices, and the
special cubic endpoint.


\begin{thebibliography}{9}

\bibitem{Bhargava}
M. Bhargava,
\emph{Galois groups of random integer polynomials and van der Waerden's
conjecture},
Ann. of Math. (2) \textbf{201} (2025), no. 2, 339--377.
\href{https://doi.org/10.4007/annals.2025.201.2.1}{doi:10.4007/annals.2025.201.2.1}.

\bibitem{BSTTTZ}
M. Bhargava, A. Shankar, T. Taniguchi, F. Thorne, J. Tsimerman,
and Y. Zhao,
\emph{Bounds on 2-torsion in class groups of number fields and integral
points on elliptic curves},
J. Amer. Math. Soc. \textbf{33} (2020), no. 4, 1087--1099.
\href{https://doi.org/10.1090/jams/945}{doi:10.1090/jams/945}.

\bibitem{ChowDietmann}
S. Chow and R. Dietmann,
\emph{Enumerative Galois theory for number fields},
preprint, \href{https://arxiv.org/abs/2304.11991v2}{arXiv:2304.11991v2},
April 3, 2026.

\bibitem{Cohen}
S. D. Cohen,
\emph{The distribution of the Galois groups of integral polynomials},
Illinois J. Math. \textbf{23} (1979), no. 1, 135--152.
\href{https://doi.org/10.1215/ijm/1256048323}{doi:10.1215/ijm/1256048323}.

\bibitem{Malle}
G. Malle,
\emph{On the distribution of Galois groups. II},
Experiment. Math. \textbf{13} (2004), no. 2, 129--135.
\href{https://doi.org/10.1080/10586458.2004.10504527}{doi:10.1080/10586458.2004.10504527}.

\end{thebibliography}
\end{document}